\documentclass[11pt,a4paper,reqno]{amsart}
\usepackage{graphicx} % Required for inserting images
\usepackage{stmaryrd}

\newcommand{\cA}{\mathcal{A}}
\newcommand{\bR}{\mathbb R}
\newcommand{\rn}{\mathbb R}
\newcommand{\del}{\nabla}
\newcommand{\cT}{\mathcal{T}}
\newcommand{\rT}{t_{\mathrm{max}}}
\newcommand{\spi}{\mathrm{Spin(7)}}
\newcommand{\qandq}{\quad\text{and}\quad}
\newcommand{\ddt}{\frac{d}{dt}}
\newcommand{\sn}{\mathbb{S}}

\DeclareMathOperator\End{End}
\DeclareMathOperator\vol{vol_g}
\DeclareMathOperator\Div{div}
\DeclareMathOperator\tr{tr}
\DeclareMathOperator\Fr{Fr}

\newtheorem{theorem}{Theorem}[section]
\newtheorem{corollary}[theorem]{Corollary}
\newtheorem{lemma}[theorem]{Lemma}
\newtheorem{proposition}[theorem]{Proposition}

\theoremstyle{definition}
\newtheorem{definition}[theorem]{Definition}
\newtheorem{remark}[theorem]{Remark}

\title{The analysis of the harmonic-$\mathrm{Spin}(7)$ flow}
\author{E. Loubeau}
\address{Univ. Brest, CNRS UMR 6205, LMBA, F-29238 Brest, France}
\email{loubeau@univ-brest.fr}
\date{April 2023}
\thanks{This work benefited from the projects CAPES-COFECUB MA 898/18, MathAmSud "Special geometries and moduli spaces" and PRCI ANR--FAPESP Bridges: ANR-21-CE40-0017/Fapesp 2021/04065-6.}

\begin{document}

\begin{abstract}
 The group $\spi$ belongs to the list of possible holonomy of an eight-dimensional Riemannian manifold. The weaker notion of $\spi$-structures plays for manifolds with holonomy $\spi$, the analogue of almost Hermitian for K{\"a}hler manifolds. As part of a more general scheme, a notion of harmonicity of $\spi$-structures is developed with the objective of comparing isometric $\spi$-structures among themselves.
 We present here an account of our study in \cite{DLSE} of the harmonic flow of $\spi$-structures and its analytical properties.
\end{abstract}

\maketitle

\section{Introduction}

At the confluence of holonomy theory and harmonic maps lies the calculus of variations for geometric structures. The catalogue of holonomy groups $\mathrm{Hol}(M,g)$ of an irreducible non-symmetric simply-connected $n$-dimensional Riemannian manifold is rather brief: $\mathrm{SO}(n)$ (generic), $\mathrm{U}(\frac{n}{2})$ (K{\"a}hler), $\mathrm{SU}(\frac{n}{2})$ (Calabi-Yau), $\mathrm{Sp}(\frac{n}{4})$ (HyperK{\"a}hler), $\mathrm{Sp}(1)\mathrm{Sp}(\frac{n}{4})$ (Quaternionic K{\"a}hler), $\mathrm{G}_2$ ($n=7)$ and $\spi$ ($n=8$).
Its main repercussion is on the symmetries of the curvature tensor which then must live in the Lie algebra $\mathfrak{Hol}(M,g)$, so that Calabi-Yau, HyperK{\"a}hler, $\mathrm{G}_2$ and $\spi$ manifolds are Ricci flat. Our interest is with the last case. As explained in \cite{Lawson-Michelsohn}, the Clifford algebra $\mathrm{Cl}_7$ is isomorphic to $\rn [8] \oplus \rn [8]$, so as the group $\spi$ lives in $\mathrm{Cl}_6$, which is isomorphic to one of these two factors, it admits one irreducible representation of dimension eight. Since it must be unitary, $\spi$ can be seen as a $21$-dimensional subgroup of $\mathrm{SO}(8)$ (cf.~\cite{vadararajan} for an expos{\'e} on its conjugacy classes).

It has been know for quite a while \cite{bonan1,bonan2,bonan3} that holonomy in $\spi$ is equivalent to the existence of a parallel $4$-form $\Phi$ point-wise equal to 
\begin{align*}
\Phi_p = & dx^{0123} - dx^{0167} - dx^{0527} -dx^{0563} - dx^{0415} - dx^{0426} - dx^{0437} + dx^{4567}  \\
& - dx^{4523} - dx^{4163} - dx^{4127} - dx^{2637} - dx^{1537} - dx^{1526} ,
\end{align*}
where $dx^{ijkl} = dx^{i}\wedge dx^{j} \wedge dx^{k} \wedge dx^{l}$ (though there exist $480$ different ways to write down this Euclidean model on $\rn^8$).

The group $\spi$ can then be thought of as the group of automorphisms of $\Phi_p$. This geometry comes from the octonions, much like $\mathrm{G}_2$, and has long been suspected to be an impostor waiting to be removed from the list, as happened to $\mathrm{Spin}(9)$ with Alekseevsky~\cite{alekseevsky}.

The first examples of Riemannian manifolds with $\spi$ holonomy are due to Bryant~\cite{bryant} in 1985 on open subsets of Euclidean spaces and complete examples followed four years later \cite{bryant-salamon} on the spinor bundle of $\sn^4$.
For compact examples, we had to wait for Joyce in 1996~\cite{joyceart}, and a comprehensive account can be found in \cite{joycebook}. Foscolo~\cite{foscolo} recently constructed complete non-compact $\spi$–manifolds with arbitrarily large second Betti number
and infinitely many distinct families of asymptotically locally conical $\spi$–metrics on the same smooth topological $M^8$. Kovalev~\cite{kovalev} adapted in 2003 a conical asymptotical gluing argument to obtain $\spi$-manifolds from twisted connected sums. More explicit is Salamon's example of the product of $\rn^+$ with the nearly-$\mathrm{G}_2$ manifold ${\mathrm{SO}(5)}/{\mathrm{SO}(3)}$.
$\spi$-Manifolds are hard to find but they are interesting for at least two other reasons:
\begin{itemize}
    \item[a)] One can define a higher gauge theory of $\spi$-instantons with the (still fairly remote) hope of defining moduli spaces and invariants, and perhaps a (partial) classification of $8$-dimensional manifolds. These ``twisted D-T instantons'' are vector bundles with a connection $A$ such that their curvature tensor $F_A$ lies in some (irreducible) component $\Omega^2_{21}$ (cf. the next section for conventions and notations), equivalently satisfies
    $$
    F_A \wedge \Phi = \star F_A .$$
    See \cite{clarke-oliveira,tanaka} for some $\spi$-instanton constructions.
    \item[b)] Supersymmetry and string theory have invested a lot in of hope in $\spi$-manifolds to construct solutions to the gravitino and dilatino equations~\cite{ivanov}.
\end{itemize}

However, all these constructions are hard but there exists the softer, more abundant, notion of $\spi$-structure.

\section{$\spi$-structures}

The best reference for this section, especially pertaining to flows, is Karigiannis' notes~\cite{karigiannis-spin7}. 

A $\spi$-structure on an $8$-dimensional manifold $M$ is a reduction of the structure group of the frame bundle $\Fr(M)$  to the Lie group $\spi\subset \mathrm{SO}(8)$. From the point of view of differential geometry, a $\spi$-structure is a $4$-form $\Phi$ on $M$. The existence of such a structure is (equivalent to) a topological condition, cf. \cite[Theorem 10.7]{Lawson-Michelsohn}: the vanishing of the first and second Stiefel-Whitney classes and, for some orientation
$$
p_1^2 - 4 p_2 + 8\chi =0.
$$
The space of $4$-forms which determine a $\spi$-structure on $M$ is a subbundle $\cA$ of $\Omega^4(M)$, called the bundle of \emph{admissible} $4$-forms. This is \emph{not} a vector subbundle and it is not even an open subbundle, unlike the case for $\mathrm{G}_2$-structures.

A $\spi$-structure determines a Riemannian metric and an orientation on $M$ in a nonlinear way. Explicit formulas can be found in \cite{karigiannis-spin7}, they are highly involved and it is hard to picture how they could be exploited. But it is crucial to our approach that several $\spi$-structures will give rise to the same Riemannian metric, much like for the ${\mathrm G}_2$-case.
The metric and the orientation determine a Hodge star operator $\star$, and the $4$-form is \emph{self-dual}, i.e., $\star \Phi=\Phi$. 

\begin{definition}
    Let $\del$ be the Levi-Civita connection of the metric $g_{\Phi}$. The pair $(M, \Phi)$ is a \emph{$\spi$-manifold} if $\del \Phi=0$. This is a non-linear partial differential equation for $\Phi$, since $\del$ depends on $g$, which in turn depends non-linearly on $\Phi$. A $\spi$-manifold has Riemannian holonomy contained in the subgroup $\spi\subset \mathrm{SO}(8)$. Such a parallel $\spi$-structure is also called \emph{torsion-free}. 
\end{definition}

\subsection{Decomposition of the space of forms}
\label{subsec:formdecomp}

The existence of a $\spi$-structure $\Phi$ induces a decomposition of the space of differential forms on $M$  into irreducible $\spi$ representations. We have the following orthogonal decomposition, with respect to $g_\Phi$:
\begin{align*}
\Omega^2=\Omega^2_{7}\oplus \Omega^2_{21},\ \ \ \ \ \ \ \Omega^3=\Omega^3_{8}\oplus \Omega^3_{48}, \ \ \ \ \ \ \ \ \  \Omega^4=\Omega^4_{1}\oplus \Omega^4_{7}\oplus \Omega^4_{27}\oplus \Omega^4_{35},
\end{align*}
where $\Omega^k_l$ has pointwise dimension $l$. Explicitly, $\Omega^2$ and $\Omega^3$ are described as follows:
\begin{align*}
    \Omega^2_7&=\{\beta \in \Omega^2 \mid \star(\Phi \wedge \beta)=-3\beta\}, \quad
    \Omega^2_{21}=\{ \beta\in \Omega^2 \mid \star(\Phi\wedge \beta)=\beta\},
\end{align*}
and
\begin{align*}
    \Omega^3_8&=\{ X\lrcorner \Phi \mid X\in \Gamma(TM)\}, \quad
    \Omega^3_{48}=\{ \gamma \in \Omega^3\mid \gamma \wedge \Phi =0\}. 
\end{align*}
In local coordinates, these spaces of forms are described as, for $\beta\in \Omega^2(M)$,
\begin{align}
    \beta_{ij}\in \Omega^2_7 \iff \beta_{ab}\Phi_{abij}&=-6\beta_{ij},\label{eq:27decomp2}\\
    \beta_{ij}\in \Omega^2_{21} \iff \beta_{ab}\Phi_{abij}&=2\beta_{ij} \label{eq:221decomp2}
\end{align}
and, for $\gamma\in \Omega^3(M)$,
\begin{align}
    \gamma_{ijk}\in \Omega^3_8 &\iff \gamma_{ijk}=X_l\Phi_{ijkl}\ \ \ \ \textup{for\ some}\ X\in \Gamma(TM), \label{eq:38decomp2}\\
    \gamma_{ijk}\in \Omega^3_{48} &\iff \gamma_{ijk}\Phi_{ijkl}=0. \label{eq:348decomp2}
\end{align}
If $\pi_7$ and $\pi_{21}$ are the projection operators on $\Omega^2$, it follows from \eqref{eq:27decomp2} and \eqref{eq:221decomp2} that 
\begin{align*}
\pi_7(\beta)_{ij}&=\frac 14\beta_{ij}-\frac 18\beta_{ab}\Phi_{abij}, \\
\pi_{21}(\beta)_{ij}&=\frac 34\beta_{ij}+\frac 18\beta_{ab}\Phi_{abij}. 
\end{align*}
Finally, for $\beta_{ij}\in \Omega^2_{21}$,
\begin{align*}
    \beta_{ab}\Phi_{bpqr}&=\beta_{pi}\Phi_{iqra}+\beta_{qi}\Phi_{irpa}+\beta_{ri}\Phi_{ipqa},   
\end{align*}
so $\Omega^2_{21}\equiv \mathfrak{so}(7)$ is the Lie algebra of $\spi$.

To describe $\Omega^4$ in local coordinates, we use the operator $\diamond$ for a $(p,q)$-tensor $\xi$ and $A\in\End(TM)$:
    \begin{align*}
        \diamond \, \xi: \quad \End(TM) &\to T^{p,q}  \\
        A\mapsto & A \diamond \xi:= \left.\ddt\right\vert_{t=0} e^{tA}.\xi. 
    \end{align*}
Now, given $A\in \Gamma(T^*M\otimes TM)$, define
\begin{align}
\label{eq:diadefn1}
    A\diamond \Phi= \frac{1}{24}(A_{ip}\Phi_{pjkl}+A_{jp}\Phi_{ipkl}+A_{kp}\Phi_{ijpl}+A_{lp}\Phi_{ijkp})dx^i\wedge dx^j\wedge dx^k\wedge dx^l,    
\end{align}
and hence 
\begin{align}
\label{eq:diadefn2}
    (A\diamond \Phi)_{ijkl}=  A_{ip}\Phi_{pjkl}+A_{jp}\Phi_{ipkl}+A_{kp}\Phi_{ijpl}+A_{lp}\Phi_{ijkp}. 
\end{align}
Recall that $
    \Gamma(T^*M\otimes TM)=\Omega^0\oplus S_0\oplus \Omega^2$,
and $\Omega^2$ splits further orthogonally, so
\begin{align*}
 \Gamma(T^*M\otimes TM)=\Omega^0\oplus S_0 \oplus \Omega^2_7\oplus \Omega^2_{21}.   
\end{align*}
With respect to this splitting, we can write $A=\frac 18 (\tr A)g+A_0+A_7+A_{21}$ where $A_0$ is a symmetric traceless $2$-tensor. 
The diamond contraction \eqref{eq:diadefn2} defines a linear map $A\mapsto A\diamond \Phi$, from $\Omega^0\oplus S_0 \oplus \Omega^2_7\oplus \Omega^2_{21}$ to $\Omega^4(M)$. 
The following proposition is proved  in \cite[Prop. 2.3]{karigiannis-spin7}.

\begin{proposition}\label{prop:diaproperties1}
The kernel of the map $A\mapsto A\diamond \Phi$ is isomorphic to the subspace $\Omega^2_{21}$. The remaining three summands $\Omega^0,\ S_0$ and $\Omega^2_7$ are mapped isomorphically onto the subspaces $\Omega^4_1,\ \Omega^4_{35}$ and $\Omega^4_7$ respectively.
\end{proposition}

To understand $\Omega^4_{27}$, we need another characterization of the space of $4$-forms using the $\spi$-structure. Following \cite{karigiannis-spin7}, we adopt the following:

\begin{definition}
    On $(M, \Phi)$, define a $\Phi$-equivariant linear operator $\Lambda_{\Phi}$ on $\Omega^4$ as follows. Let $\sigma\in \Omega^4(M)$ and let $(\sigma \cdot \Phi)_{ijkl}=\sigma_{ijmn}\Phi_{mnkl}$. Then 
\begin{align*}
    (\Lambda_{\Phi}(\sigma))_{ijkl}=(\sigma \cdot \Phi)_{ijkl}+(\sigma \cdot \Phi)_{iklj}+(\sigma \cdot \Phi)_{iljk}+(\sigma \cdot \Phi)_{jkil}+(\sigma \cdot \Phi)_{jlki}+(\sigma \cdot \Phi)_{klij}. 
\end{align*}
\end{definition}

\begin{proposition}
\label{prop:427decomp}
    The spaces $\Omega^4_1$, $\Omega^4_7,\ \Omega^4_{27}$ and $\Omega^4_{35}$ are all eigenspaces of $\Lambda_{\Phi}$ with distinct eigenvalues:
\begin{align*}
  \begin{aligned}
   \Omega^4_1&=\{\sigma\in \Omega^4 \mid \Lambda_{\Phi}(\sigma)=-24\sigma\}, \\       \Omega^4_{27}&=\{\sigma\in \Omega^4 \mid \Lambda_{\Phi}(\sigma)=4\sigma\},
  \end{aligned}
  &&
  \begin{aligned}
   \Omega^4_7&=\{\sigma\in \Omega^4 \mid \Lambda_{\Phi}(\sigma)=-12\sigma\}, \\       \Omega^4_{35}&=\{\sigma\in \Omega^4 \mid \Lambda_{\Phi}(\sigma)=0\}.
  \end{aligned}
 \end{align*}
 Moreover, the decomposition of $\Omega^4(M)$ into self-dual and anti-self-dual parts is
 \begin{align*}
\Omega^4_+=\{\sigma \in \Omega^4\mid \star \sigma = \sigma\}=\Omega^4_1\oplus \Omega^4_7\oplus \Omega^4_{27},\ \ \ \ \ \Omega^4_-=\{\sigma\in \Omega^4\mid \star \sigma=-\sigma\}=\Omega^4_{35} . 
 \end{align*}
\end{proposition}

\medskip

Before we discuss the torsion of a $\spi$-structure, we note some contraction identities involving the $4$-form $\Phi$. In local coordinates $\{x^1, \cdots, x^8\}$, the $4$-form $\Phi$ is
\begin{align*}
\Phi=\frac{1}{24}\Phi_{ijkl}\ dx^i\wedge dx^j\wedge dx^k\wedge dx^l    
\end{align*}
where $\Phi_{ijkl}$ is totally skew-symmetric. We have the following identities, as always summing on repeated indices, which encapsulate the symmetries of a $\spi$-structure
\begin{align}
    \Phi_{ijkl}\Phi_{abcl}
    &=g_{ia}g_{jb}g_{kc}+g_{ib}g_{jc}g_{ka}+g_{ic}g_{ja}g_{kb}\nonumber \\
    & \quad -g_{ia}g_{jc}g_{kb}-g_{ib}g_{ja}g_{kc}-g_{ic}g_{jb}g_{ka} \nonumber \\
    & \quad -g_{ia}\Phi_{jkbc}-g_{ib}\Phi_{jkca}-g_{ic}\Phi_{jkab} \nonumber \\
    & \quad -g_{ja}\Phi_{kibc}-g_{jb}\Phi_{kica}-g_{jc}\Phi_{kiab} \nonumber \\
    & \quad -g_{ka}\Phi_{ijbc}-g_{kb}\Phi_{ijca}-g_{kc}\Phi_{ijab} \label{eq:impiden1} \\
    \Phi_{ijkl}\Phi_{abkl}&=6g_{ia}g_{jb}-6g_{ib}g_{ja}-4\Phi_{ijab} \label{eq:impiden2} \\
    \Phi_{ijkl}\Phi_{ajkl}&=42g_{ia} \label{eq:impiden3} \\
    \Phi_{ijkl}\Phi_{ijkl}&=336 . \label{eq:impiden4}
\end{align}
We also have contraction identities involving $\del \Phi$ and $\Phi$
\begin{align*}
(\del_m\Phi_{ijkl})\Phi_{abkl}&=-\Phi_{ijkl}(\del_m\Phi_{abkl})-4\del_m\Phi_{ijab} \\
(\del_m\Phi_{ijkl})\Phi_{ajkl}&=-\Phi_{ijkl}(\del_m\Phi_{ajkl}) \\
(\del_m\Phi_{ijkl})\Phi_{ijkl}&=0. =
\end{align*}

We now describe the \emph{torsion} of a $\spi$-structure. Given $X\in \Gamma(TM)$, we know from \cite[Lemma 2.10]{karigiannis-spin7} that $\del_X\Phi$ lies in the subbundle $\Omega^4_7\subset\Omega^4$. 
\begin{definition}
    The \emph{torsion tensor} of a $\spi$-structure $\Phi$ is the element of $\Omega^1_8\otimes \Omega^2_7$ defined by expressing $\del \Phi$ in the light of Proposition \ref{prop:diaproperties1}:
\begin{align}
\label{Tdefneqn}
    \del_m\Phi_{ijkl}=(T_m\diamond \Phi)_{ijkl}=T_{m;ip}\Phi_{pjkl}+T_{m;jp}\Phi_{ipkl}+T_{m;kp}\Phi_{ijpl}+T_{m;lp}\Phi_{ijkp}    
\end{align}
    where $T_{m;ab}\in\Omega^2_7$, for each fixed $m$. 
\end{definition}

Directly in terms of $\del \Phi$, the torsion $T$ is given by
\begin{align}
\label{eq:Texpress}
    T_{m;ab} 
    =\frac{1}{96}(\del_m\Phi_{ajkl})\Phi_{bjkl}    
\end{align}

{\begin{remark}
We remark that the notation $T_{m;ab}$ \emph{should not} be confused with taking two covariant derivatives of $T_m$. The torsion tensor $T$ is an element of $\Omega^1_8\otimes \Omega^2_7$ and thus for each fixed index $m$, $T_{m;ab}\in \Omega^2_7$, but $T$ is not in $\Omega^3$. 
\end{remark}}

\begin{theorem}
\cite{fernandez-spin7}
    The $\spi$-structure $\Phi$ is torsion-free if, and only if, $d\Phi=0$. Since $\star \Phi=\Phi$, this is equivalent to $d^*\Phi=0$.
\end{theorem}

Finally, the torsion  satisfies a `Bianchi-type identity'. This was first proved in \cite[Theorem 4.2]{karigiannis-spin7}, using the diffeomorphism invariance of the torsion tensor. A different proof can be found in~\cite[Theorem 3.9]{DLSE}, using the Ricci identity
\begin{equation*} 
    \nabla_{k} \nabla_{i} X_l - \nabla_{i} \nabla_{k} X_l 
    = - R_{kilm} X_m.
\end{equation*}

\begin{theorem}\label{thm:spin7bianchi}
The torsion tensor $T$ satisfies the following `Bianchi-type identity'
\begin{align}
\label{spin7bianchi}
    \del_iT_{j;ab}-\del_jT_{i;ab}=2T_{i;am}T_{j;mb}-2T_{j;am}T_{i;mb}+\frac 14R_{jiab}-\frac 18R_{jimn}\Phi_{mnab}.  
\end{align}
\end{theorem}

Using the Riemannian Bianchi identity, we see  that
\begin{align*}
R_{ijkl}\Phi_{ajkl}=-(R_{jkil}+R_{kijl})\Phi_{ajkl}=-R_{iljk}\Phi_{aljk}-R_{ikjl}\Phi_{akjl},    
\end{align*}
hence 
\begin{align*}
    R_{ijkl}\Phi_{ajkl}=0.    
\end{align*}
Using this and contracting \eqref{spin7bianchi} on $j$ and $b$ gives the expression for the Ricci curvature of a metric induced by a $\spi$-structure:
\begin{align}
\label{ricci}
    R_{ij}=4\del_iT_{a;ja}-4\del_aT_{i;ja}-8T_{i;jb}T_{a;ba}+8T_{a;jb}T_{i;ba}.   
\end{align}
This also proves that the metric of a torsion-free $\spi$-structure is Ricci-flat, a result originally due to Bonan~\cite{bonan1}. Taking the trace of \eqref{ricci} gives the  scalar curvature $R$:
\begin{align*}
    R=4\del_iT_{a;ia}-4\del_aT_{i;ia}+8|T|^2+8T_{a;jb}T_{j;ba}.    
\end{align*}

\begin{remark}
\begin{enumerate}
\item A classification of $\spi$-structures was given by Fernandez in~\cite{fernandez-spin7} and a formulation in terms of spinors can be found in \cite{Martin-Merchan}. 
\item Compact simply-connected Riemannian symmetric spaces cannot carry any invariant $\spi$-structures and the compact simply-connected almost effective homogeneous space with invariant $\spi$-structures are ${\mathrm{SU}(3)}/\{e\}$, some torus bundles over 
$\Big({\mathrm{SU}(2)}/{\mathrm{U}(1)}\Big)^{\times 3}$ and the Calabi-Eckmann ${\mathrm{SU}(3)}/{\mathrm{SU}(2)}\times \mathrm{SU}(2)$.
\item Without requiring invariance of the structure, the $8$-dimensional compact simply-connected Riemannian symmetric spaces admitting a $\spi$-structures are $\mathrm{SU}(3)$, $\sn^3\times\sn^3\times\sn^2$, $\sn^5\times\sn^3$, $\mathbb{HP}^2$, $\mathrm{Gr}_{2}({\mathbb C}^4)$ and the Wolf space ${{\mathrm G}_2}/{\mathrm{SO}(4)}$~\cite{ACFR}.
\end{enumerate}
\end{remark}

\section{Harmonicity}

The ultimate goal in $\mathrm{Spin}(7)$-geometry is to find parallel structures. Not only is it quite a difficult task involving a non-linear equation and hard analysis but topological obstructions also apply.

An alternative strategy to finding the best among all possible $\mathrm{Spin}(7)$-structures is to introduce a variational problem, for example measuring the default of parallelism, and search for minimisers.

This is the junction point between $\mathrm{Spin}(7)$-geometry and harmonic map theory, though the price to pay is we need to fix the metric, i.e. work within the isometric class of $\mathrm{Spin}(7)$-structures.

\begin{definition}
Two $\mathrm{Spin}(7)$-structures $\Phi_1$ and $\Phi_2$ on $M$ are called \emph{isometric} if they induce the same Riemannian metric, that is, if $g_{\Phi_1}=g_{\Phi_2}$. We will denote by $\llbracket \Phi\rrbracket$ the space of $\mathrm{Spin}(7)$-structures that are isometric to a given $\mathrm{Spin}(7)$-structure $\Phi$.
\end{definition}

\begin{definition}
Let $\Phi_0$ be a fixed initial $\mathrm{Spin}(7)$-structure on $M$. The \emph{energy functional} $E$ on the set $\llbracket \Phi_0 \rrbracket$ is
\begin{align}
\label{energyfuncdefn}
    E(\Phi)=\frac 12 \int_M |T_{\Phi}|^2 \vol_{g_{\Phi}} ,
\end{align}
where $T_{\Phi}$ is the torsion of $\Phi$.
\end{definition}

Once the variational problem has been delineated, the next step is to derive the corresponding Euler-Lagrange equation. We will call critical points of $E\large|_{\llbracket \Phi_0 \rrbracket}$ {\em harmonic $\mathrm{Spin}(7)$-structures} and work out the harmonic equation for $\mathrm{Spin}(7)$-structures and the corresponding isometric flow.

The main ingredient is the representation theory properties outlined in the previous section and the recipe is to follow the treatment of the $\mathrm{G}_2$ case in \cite[Section 6]{LSE}, only slightly adapted to specific properties of $\mathrm{Spin}(7)$-geometry.

The link with harmonic map theory is the one-one correspondence between $\spi$-structures $\Phi$ and sections $\sigma$ of an ad-hoc $\spi$-twistor bundle $N$, constructed as the $\spi$ quotient of the $\mathrm{SO}(8)$ frame bundle of $(M^8,g)$. The fibres are isometric to $\mathbb{RP}^7$ and parametrise isometric $\spi$-structures on $(M^8,g)$.

To obtain the equation of harmonicity, one must first and foremost identify the tangent space of fibres in order to be able to consider vertical variations and compare (isometric) $\spi$-structures among themselves.

The first constituent is the connection form $f$, which identifies the vertical of the tangent bundle of the ``twistor space'' with $\mathfrak{m}$ the (naturally reductive) complement of $\mathfrak{so}(7) (= \mathfrak{spin}(7))$ in $\mathfrak{so}(8)$. Sections of this space correspond to $\mathrm{Spin}(7)$-structures and restricting ourselves to the vertical part means we only look at variations through $\mathrm{Spin}(7)$-structures.

If $\tilde{\Phi}$ is the universal $\mathrm{Spin}(7)$-structure, a sort of ideal $\spi$-structure living a couple of fibre bundles above the manifold $M$ (cf. \cite{LSE} for particulars), then the connection form is characterised by
$$
\nabla_A \tilde{\Phi} = f(A).\tilde{\Phi} .
$$
Here $f(A)$ is in $\mathfrak{m}$.

We identify $\mathfrak{so}(8)$ with $\Omega^2$ and $\mathfrak{m}$ is then identified with $\Omega^2_7$.

Since $\tilde{\Phi}$ is in $\Omega^4$ (of the appropriate space), the term $f(A).\tilde{\Phi}$ should be understood as the diamond operator of Equation~\eqref{eq:diadefn1}, which is just the derivation of the natural action of $\mathrm{GL}(8)$ by pulling back forms.

To obtain $f$ we need to find an inverse of the $\diamond$ operator and to do this introduce the triple contraction $\lrcorner_3$ between two four-forms (we follow notations and conventions of \cite{karigiannis-spin7}):

If $\beta = \tfrac12 \beta_{ij} dx^i \wedge dx^j$ then $\beta \diamond \Phi \in \Omega^4$ and put $(\beta \diamond \Phi) \lrcorner_3 \Phi$ to be the two-form defined by 
$$
(\beta \diamond \Phi) \lrcorner_3 \Phi = \frac12 ((\beta \diamond \Phi) \lrcorner_3 \Phi)_{pq} dx^p \wedge dx^q ,
$$
where
$$
((\beta \diamond \Phi) \lrcorner_3 \Phi)_{pq} = 
(\beta \diamond \Phi)_{pijk} \Phi_{qijk} .
$$
Because we are interested in the case $\beta = f(A) \in \Omega^2_7$, we can use  $\beta_{ab}\Phi_{abij} = -6 \beta_{ij}$ to compute that
$$
(\beta \diamond \Phi) \lrcorner_3 \Phi = 96 \beta.
$$

Once we have this, the rest follows relatively easily, if one knows where to pick information in \cite{karigiannis-spin7}:
\begin{itemize}
    \item The connection form is then given by
$$
96 f(A) = \nabla_A \tilde{\Phi} \lrcorner_3 \tilde{\Phi}
$$
and, since $\mathrm{Spin}(7)$-structures $\Phi$ and sections $\sigma: M \to N$ of the twistor space are related by $\Phi = \tilde{\Phi}\circ\sigma$, we can pull back the above formula to obtain
$$
f(d\sigma(X)) = \tfrac{1}{96} (\nabla_{X}\Phi) \lrcorner_3 \Phi
$$
which is precisely the torsion $T(X)$ of \eqref{eq:Texpress} in the space $\Omega^2_7$.
\item The (vertical) energy density of the section $\sigma: M \to $ N is
$$
|d^v\sigma|^2 = |T|^2,
$$
so the functional we take, the $L^2$-norm of the torsion, is exactly the Dirichlet energy of $\sigma$ (at least up to a constant due to the contribution of the horizontal part).
\item The vertical tension field is 
$$
I(\tau^v(\sigma)) = \sum_1^8 \nabla_{e_i} (T(e_i)) - T(\nabla_{e_i} e_i) = \Div T .
$$
\item The flow of sections $\sigma_t : M \to N$
$$
\frac{d\sigma_t}{dt} = \tau^v(\sigma_t)
$$  
is equivalent to
$$
I(\frac{d\sigma_t}{dt}) = I(\tau^v(\sigma_t)) ,
$$
where $I$ plays the role of an extended $f$.
We know that $I(\tau^v(\sigma_t))= \Div T_t$ and, generalising to $M\times \bR$ (or at least on an interval) all the previous objects, we have that 
$$
I(\frac{d\sigma_t}{dt}) = \tfrac{1}{96} \frac{d{\Phi}_t}{dt} \lrcorner_3 {\Phi}_t .
$$
On the other hand, since $\Div T_t$ is in $\Omega^2_7$
$$
\Div T_t = \tfrac{1}{96} (\Div T_t \diamond {\Phi}_t) \lrcorner_3 {\Phi}_t ,
$$
and $\lrcorner_3 {\Phi}_t$ is an isomorphism on $\Omega^2_7$ (its kernel is $\Omega^2_{21}$), we have the isometric flow, with initial value:
\begin{align} 
\label{eq: Har Spin(7) Flow} 
 \left\{\begin{array}{rl} 
      & \frac{d \Phi}{d t} = \Div T \diamond \Phi \\
      & \Phi(0) =\Phi_0 .
      \tag{HF}
   \end{array}\right.
\end{align}
\end{itemize}

\begin{remark}
\begin{itemize}
\item The $\Div T$ equation is the vertical part of the harmonic map equation of $\sigma$, which is known to admit short-time existence, so this property carries over to our heat flow.
\item As (the fibres of) the target are isometric to the real seven-dimensional projective space, they have positive sectional curvature, so there can be no certainty about the long-time existence of the flow (cf. \cite{ES}).
\item Solitons of such flows are studied for a general group $H$ in \cite{DLSE} and \cite{FLMSE}.
\end{itemize}
\end{remark}

\section{Analysis of the flow I}

This section develops tools for the analysis of the isometric flow of $\spi$-structures. Some proofs of the statements in this section have appeared in full in \cite{DLSE} and we refer to them. Others were consequences of more general arguments and here we present their $\spi$ versions, though they are only adaptations of their $\mathrm{G_2}$ counterparts found in~\cite{DGK}.

Let $\{ \partial_t, e_1, \dots , e_7\}$ be an orthonormal (geodesic) frame. First, we use the formula
$$
(R(e_i,e_j)T)(e_a,e_b,e_c) = - T(R(e_i,e_j)e_a,e_b,e_c) - T(e_a,R(e_i,e_j)e_b,e_c) - T(e_a,e_b,R(e_i,e_j)e_c)
$$
to derive a formula for the Laplacian of the torsion of a $\mathrm{Spin}(7)$-structure.

\begin{lemma} \cite[Lemma 4.12]{DLSE}
Let $\Delta = \tr{\nabla_{e_i}\nabla_{e_i}}$ be the Laplacian, then
 \begin{align*}
   (\Delta T)_{m;ab} =& \nabla_m \nabla_i T_{i;ab} - T_{q;ab}R_{imiq} - T_{i;qb}R_{imaq} - T_{i;aq}R_{imbq} + 2 \nabla_i T_{i;ap}T_{m;bp} + 2 T_{i;ap}\nabla_i T_{m;bp} \\
    &  - 2 \nabla_i T_{m;ap}T_{i;bp} - 2 T_{m;ap}\nabla_i T_{i;pb} + \tfrac14 \nabla_i R_{miab} - \tfrac18 \nabla_i R_{mipq} \Phi_{pqab} - \tfrac18 R_{mipq} \nabla_i \Phi_{pqab}.
\end{align*}
\end{lemma}
This allows us to compute a local expression for the evolution of the torsion $T$.
\begin{proposition}\cite[Proposition 4.13]{DLSE}
\label{prop: evolution of T}
    Let $\{\Phi_t\}$ be a solution of the harmonic $\spi$-flow \eqref{eq: Har Spin(7) Flow}, then its torsion evolves according to the equation
\begin{align*}
    4\frac{\partial}{\partial t} T_{m;is} 
    =&\ 
    4(\Delta T)_{m;is} \\
    &+ \nabla_a T_{m;bc}\Big(4T_{a;bp}\Phi_{pcis}+T_{a;ip}\Phi_{bcps}+T_{a;sp}\Phi_{bcip} \Big) + T_{m;bc} \nabla_a T_{a;bp} \Phi_{pcis} \\
    &+ 3\nabla_a T_{a;ip} T_{m;ps} + \nabla_a T_{a;sp} T_{m;pi} - 2 T_{a;ip}\nabla_a T_{m;sp} + 2 \nabla_a T_{m;ip}T_{a;sp} \\
    & + T_{m;bc} T_{a;bp} \Big(T_{a;pq}\Phi_{qcis}+ T_{a;cq}\Phi_{pqis}+2T_{a;iq}\Phi_{pcqs} +
    2T_{a;sq}\Phi_{pciq} \Big) \\
    &+  \tfrac12 T_{m;bc}T_{a;ip} \Big(
    T_{a;pq}\Phi_{bcqs}+
    2T_{a;sq}\Phi_{bcpq} \Big) 
    + \tfrac12 T_{m;bc} T_{a;sp}
    T_{a;pq}\Phi_{bciq} \\
    & + 4 T_{q;is}R_{amaq}
    - (\nabla_a R_{mais} - \tfrac12 \nabla_a R_{mapq} \Phi_{pqis})\\
    &+ T_{a;qs}R_{amiq} + T_{a;iq}R_{amsq}  
    + \tfrac18 R_{mapq} \nabla_a \Phi_{pqis}  - T_{a;qc}R_{ambq}\Phi_{bcis} 
    - \tfrac{1}{16} R_{mapq} \nabla_a \Phi_{pqbc}\Phi_{bcis}.
\end{align*}
\end{proposition}
But the real information is the evolution of the norm of the torsion.
\begin{proposition}\cite[Proposition 4.14]{DLSE}
\label{prop:evolnormT}
    If $\{\Phi_t\}$ is a solution of the harmonic $\mathrm{Spin}(7)$-flow \eqref{eq: Har Spin(7) Flow}, then the evolution equation for $|T|^2$ is
\begin{align*}
  2\frac{\partial}{\partial t} |T|^2 &= 2\Delta |T|^2 - 4 |\nabla T|^2 
   + 16 T_{a;bp} T_{m;bc} T_{a;pq} T_{m;qc} 
   + 16 T_{a;bp} T_{m;bc} T_{a;cq} T_{m;pq} \\
  & \quad  + 16 T_{a;qs} T_{m;is} R_{amiq} 
  + 4 T_{q;is} T_{m;is} R_{amaq} 
  -4 T_{m;is} \nabla_a R_{mais}. 
\end{align*}
\end{proposition}

Both the doubling-time estimate and Shi-type estimates can be derived from general properties of the harmonic flow of $H$-structures (cf.~\cite{FLMSE}) and do not feature in \cite{DLSE} but proofs specific to $\spi$ can be written.

\begin{lemma}\cite[Corollary 4.9]{DLSE}
 There exists $\delta >0$ such that
 $$
 \cT(t) \leq 2 \cT(0)
 $$
 for all $0\leq t \leq \delta$, where 
 $$
 \cT(t) = \sup_{M} |T(x,t)|
 $$
\end{lemma}
\begin{proof}
We follow the arguments of \cite[Proposition 3.2]{DGK} and adapt them to the group $\spi$.

Wlog, we can assume that $|T| > 1$. Then
\begin{align*}
\frac{\partial}{\partial t} |T|^2 &= 2 \langle  \Delta T , T \rangle + 2 \langle\nabla T * T * \Phi , T \rangle + 2 \langle T*T*T , T \rangle + 2 \langle T*\nabla R * R * \Phi , T \rangle \\
&= \Delta |T|^2 - 2 |\nabla T|^2 + 2 \langle\nabla T * T * \Phi , T \rangle + 2 \langle T*T*T , T \rangle + 2 \langle T*\nabla R * R * \Phi , T \rangle \\
&\leq  \Delta |T|^2 - 2 |\nabla T|^2 + C |\nabla T| |T|^2  + C|T|^4 + C|T|^2 + C|T|
\end{align*}
because of the bounded geometry.

Use Young Inequality $ab \leq \tfrac{1}{2\epsilon} a^2 + \tfrac{\epsilon}{2} b^2$ to get rid of the term $|\nabla T| |T|^2$:
\begin{align*}
\frac{\partial}{\partial t} |T|^2 
&\leq  \Delta |T|^2 + (- 2 + \tfrac{C}{2\epsilon})  |\nabla T|^2 + C (1 + \tfrac{\epsilon}{2})|T|^4 + C|T|^2 + C|T| ,
\end{align*}
with $\epsilon$ large enough to ensure that $(- 2 + \tfrac{C}{2\epsilon}) < 0$.

Then, using $|T|>1$, we obtain a formula similar to \cite[(3.10)]{DGK}
\begin{align} \label{bochner}
\frac{\partial}{\partial t} |T|^2 
&\leq  \Delta |T|^2 + (- 2 + \tfrac{C}{2\epsilon})  |\nabla T|^2 + C (1 + \tfrac{\epsilon}{2})|T|^4 + C|T|^2,
\end{align}
and argue as in \cite[page 22]{DGK} with the maximal principle to get the DTE.
\end{proof}

Shi-type estimates are crucial at several steps of our various arguments in the next section. They essentially control higher-derivatives from a bounds on the (norm of the) torsion and the geometry of the manifold.
A much more general version of these Shi-type estimates can be found in \cite{DLSE}. 

\begin{lemma}[Shi-type estimates] \cite[corollary 4.10]{DLSE}
There exist constants $C_m$ such that if
$$
|T| \leq K \text{ and } |\nabla^j R|\leq B_j K^{2+j}
$$ on $M \times [0,1/K^2]$ then
$$
|\nabla^m T|\leq C_m t^{-\tfrac{m}{2}}K .
$$
\end{lemma}
\begin{remark} Note that this version has a conclusion valid over an interval slightly larger than in \cite{DLSE}, up to $1/K^2$ instead of $1/K^4$, but this has no bearing on the issue.
\end{remark}
\begin{proof}
We closely follow the proof by induction in \cite{DGK}, mutatis mutandis, and only indicate the key steps and differences. We use the symbol $*$ to denote various tensor contractions, the precise form of which is unimportant.

The base case of the induction:

We start with the evolution equation for $\nabla T$:
\begin{align*}
&\frac{\partial}{\partial t} \nabla T = \Delta \nabla  T + \nabla (\nabla T * T * \Phi) + \nabla T*T*T* \Phi +T*T*T* \nabla \Phi + \nabla T*R  + T*\nabla R \\
&+ \nabla^2 R * \Phi + \nabla R * \nabla\Phi + \nabla R *\nabla \Phi * \Phi + R *\nabla^2 \Phi * \Phi  + R *\nabla \Phi * \nabla \Phi+ \nabla T*R*\Phi \\
&+  T*\nabla R*\Phi +  T* R* \nabla\Phi ,
\end{align*}
therefore
\begin{align*}
&\frac{\partial}{\partial t} |\nabla T|^2 = \Delta |\nabla T|^2 - 2|\nabla^2 T|^2 \\
&+ 2 \langle \nabla T, \nabla (\nabla T * T * \Phi) + \nabla T*T*T* \Phi +T*T*T* \nabla \Phi + \nabla T*R  + T*\nabla R + \nabla^2 R * \Phi \\
&+ \nabla R * \nabla\Phi + \nabla R *\nabla \Phi * \Phi + R *\nabla^2 \Phi * \Phi  + R *\nabla \Phi * \nabla \Phi+ \nabla T*R*\Phi+  T*\nabla R*\Phi \\
&+  T* R* \nabla\Phi \rangle ,
\end{align*}
therefore
\begin{align*}
&\frac{\partial}{\partial t} |\nabla T|^2 \leq \Delta |\nabla T|^2 - 2|\nabla^2 T|^2 + 2 \langle \nabla T, \nabla (\nabla T * T * \Phi)\rangle \\
& + C |\nabla T|^2|T|^2 + C |\nabla T||T|^4 + C |\nabla T|^2|R| + C |\nabla T||T||\nabla R| + C |\nabla T||\nabla^2 R| + C |\nabla T||T|^2 |R| .
\end{align*}
Since, by assumption, $|R| \leq B_0 K^2$, $|\nabla R|\leq B_1 K^3$, $|\nabla^2 R|\leq B_2 K^4$ and $|T| \leq K$ we have 
\begin{align*}
&\frac{\partial}{\partial t} |\nabla T|^2 \leq \Delta |\nabla T|^2 - 2|\nabla^2 T|^2 + 2 \langle \nabla T, \nabla (\nabla T * T * \Phi)\rangle  + C K^2|\nabla T|^2 + C K^4|\nabla T| .
\end{align*}
As
\begin{align*}
&\langle \nabla T, \nabla (\nabla T * T * \Phi)\rangle \leq CK |\nabla T||\nabla^2 T| + C |\nabla T|^3 + C K^2 |\nabla T|^2 ,
\end{align*}
with Young Inequality we have
$$
2 CK |\nabla T||\nabla^2 T| \leq \tfrac{CK^2}{\epsilon}|\nabla T|^2 + C\epsilon|\nabla^2 T|^2 ,
$$
and 
\begin{align}\label{ineq}
&\frac{\partial}{\partial t} |\nabla T|^2 \leq \Delta |\nabla T|^2 - (2 - C\epsilon) |\nabla^2 T|^2   + C K^2|\nabla T|^2 + C K^4|\nabla T| + C |\nabla T|^3 .
\end{align}
The problem lies with the $|\nabla T|^3$ term.

In local coordinates the expression of $4(\nabla T*T*\Phi)$ is
\begin{align*}
&4 (\nabla T * T * \Phi)_{m;is} = \nabla_a T_{m;bc} \Big(4T_{a;bp}\Phi_{pcis}+T_{a;ip}\Phi_{bcps}+T_{a;sp}\Phi_{bcip} \Big) + T_{m;bc} \nabla_a T_{a;bp} \Phi_{pcis} \\
&+ 3\nabla_a T_{a;ip} T_{m;ps} + \nabla_a T_{a;sp} T_{m;pi} - 2 T_{a;ip}\nabla_a T_{m;sp} + 2 \nabla_a T_{m;ip}T_{a;sp}  ,
\end{align*}
so the terms making up $|\nabla T|^3$ are
\begin{align*}
&i) 4 \nabla_a T_{m;bc}\nabla_k T_{a;bp}\Phi_{pcis} \nabla_k T_{m;is} ;
\\
&ii) \nabla_a T_{m;bc}\nabla_k T_{a;ip}\Phi_{bcps} \nabla_k T_{m;is} ;
\\
&iii) \nabla_a T_{m;bc}\nabla_k T_{a;sp}\Phi_{bcip} \nabla_k T_{m;is} ;
\\
&iv) \nabla_k T_{m;bc} \nabla_a T_{a;bp} \Phi_{pcis} \nabla_k T_{m;is} ;
\\
&v) 3\nabla_a T_{a;ip} \nabla_k T_{m;ps} \nabla_k T_{m;is} ;
\\
&vi) \nabla_a T_{a;sp} \nabla_k T_{m;pi} \nabla_k T_{m;is} ;
\\
&vii) - 2 \nabla_k T_{a;ip}\nabla_a T_{m;sp} \nabla_k T_{m;is} ;
\\
&viii) 2 \nabla_a T_{m;ip} \nabla_k T_{a;sp} \nabla_k T_{m;is} .
\end{align*}
Using skew-symmetry and the Bianchi-type identity of $\Phi$, the terms ii) and vii) can be re-written:

\begin{align*}
   ii) \nabla_a T_{m;bc}\nabla_k T_{a;ip}\Phi_{bcps} \nabla_k T_{m;is} &=
    \tfrac12 \nabla_k T_{a;ip}\nabla_k T_{m;is}\Phi_{bcps} (\nabla_a T_{m;bc} - \nabla_m T_{a;bc})\\
    &= \nabla T * \nabla T * T*T*\Phi + \nabla T * \nabla T * R*\Phi*\Phi ;\\
    vii) - 2 \nabla_k T_{a;ip}\nabla_a T_{m;sp} \nabla_k T_{m;is} &= 
    - \nabla_k T_{a;ip}\nabla_k T_{m;is} (\nabla_a T_{m;sp} - \nabla_m T_{a;sp})\\
    &= \nabla T * \nabla T * T*T + \nabla T * \nabla T * R*\Phi .
\end{align*}

Exchanging $i$ and $s$ we have that the term iii) equals ii) and the term viii) equals the term vii), while vi) is the opposite of v). 

Since $T_{m;is}$ is in $\Omega^2_7$, the term iv) can be re-written

\begin{align*}
    iv) \nabla_k T_{m;bc} \nabla_a T_{a;bp} \Phi_{pcis} \nabla_k T_{m;is} &= 
    \nabla_k T_{m;bc} \nabla_a T_{a;bp} ( \nabla_k (\Phi_{pcis} T_{m;is}) - \nabla_k\Phi_{pcis}  T_{m;is}) \\
   &=  \nabla_k T_{m;bc} \nabla_a T_{a;bp} ( -6 \nabla_k (T_{m;pc}) - \nabla_k\Phi_{pcis}  T_{m;is}) \\
    &= - \nabla_k T_{m;bc} \nabla_a T_{a;bp} \nabla_k\Phi_{pcis}  T_{m;is}\\
    &= \nabla T * \nabla T * T * \nabla \Phi .
\end{align*}

We do a similar thing to the first term (forgetting the factor 4):
\begin{align*}
    i) \nabla_a T_{m;bc}\nabla_k T_{a;bp}\Phi_{pcis} \nabla_k T_{m;is} &=
    \nabla_a T_{m;bc}\nabla_k T_{a;bp} (\nabla_k (\Phi_{pcis} T_{m;is}) - \nabla_k \Phi_{pcis} T_{m;is})\\
    &= -6 \nabla_a T_{m;bc}\nabla_k T_{a;bp}\nabla_k T_{m;pc} + \nabla T * \nabla T * T *T ,
    \end{align*}
and exchanging a and m and then b and c, we have
\begin{align*}
2\nabla_a T_{m;bc}\nabla_k T_{a;bp}\nabla_k T_{m;pc} &= 
\nabla_k T_{a;bp}\nabla_k T_{m;pc} (\nabla_a T_{m;bc} - \nabla_m T_{a;bc}) ,
\end{align*}
so we can use the Bianchi-type equality again.

In conclusion, the terms leading to the problematic term $|\nabla T|^3$ can be re-written in terms of the type:
$$ \nabla T * \nabla T * T * T + \nabla T * \nabla T * R * \Phi + \nabla T * \nabla T * T * T* \Phi + \nabla T * \nabla T * R * \Phi * \Phi + \nabla T * \nabla T * T * \nabla \Phi ,$$
and 
$$|\nabla T|^3 \leq C K^2 |\nabla T|^2 ,$$
so, for a suitable $\epsilon$, Inequality~\eqref{ineq} becomes
\begin{align*}
&\frac{\partial}{\partial t} |\nabla T|^2 \leq \Delta |\nabla T|^2  + C K^2|\nabla T|^2 + C K^4|\nabla T| ,
\end{align*}
which is exactly equation (3.22) in \cite{DGK}.

For the function $f = t |\nabla T|^2 + \beta |T|^2$, combining results for $\frac{\partial}{\partial t} |\nabla T|^2$ and $\frac{\partial}{\partial t} |T|^2$, keeping in mind that $t\leq 1/K^2$ and choosing $\beta$ large enough, this implies that
\begin{align*}
    \frac{\partial}{\partial t} f &= |\nabla T|^2 + t \frac{\partial}{\partial t} |\nabla T|^2 + \beta \frac{\partial}{\partial t} |T|^2 \\
    & \leq \Delta f + C\beta K^4 .
\end{align*}
As $f(x,0) = \beta |T|^2 \leq \beta K^2$
then $\sup_{M} f(x,t) \leq CK^2 + C\beta t K^4  \leq CK^2$ hence $t|\nabla T|^2 \leq CK^2$.

The m-step of the induction:

Assume that $|\nabla^j T| \leq C_j K t^{-\tfrac j2}$ for $j=1,\dots, m-1$.

The evolution equation for $|\nabla^m T|^2$ is

\begin{align*}
&\frac{\partial}{\partial t} \nabla^m T = \Delta \nabla^m  T + \sum_{i=0}^m \nabla^{m-i} T * \nabla^i R +\nabla^m(\nabla T * T * \Phi) + \sum_{a+b+c+d =m} \nabla^a T*\nabla^b T* \nabla^c T* \nabla^d \Phi \\
&+ \sum_{i=0}^m \nabla^i T* \nabla^{m-i} R  + \sum_{i=0}^m  \nabla^{m+1-i} R *\nabla^i \Phi + \sum_{a+b+c=m} \nabla^a R * \nabla^{b+1}\Phi* \nabla^{c}\Phi \\
&+ \sum_{a+b+c=m} \nabla^a T * \nabla^{b}R * \nabla^{c}\Phi ,
\end{align*}
therefore
\begin{align*}
&\frac{\partial}{\partial t} |\nabla^m T|^2 = \Delta |\nabla^m T|^2 - 2|\nabla^{m+1} T|^2 
+ \sum \nabla^{m} T *\nabla^{m-i} T * \nabla^i R + \nabla^{m} T*\nabla^m(\nabla T * T * \Phi) \\
&+\sum \nabla^{m} T *\nabla^a T*\nabla^b T* \nabla^c T* \nabla^d \Phi 
+ \sum \nabla^{m} T *\nabla^i T* \nabla^{m-i} R  + \sum \nabla^{m} T *\nabla^{m+1-i} R *\nabla^i \Phi \\
& + \sum \nabla^{m} T *\nabla^a R * \nabla^{b+1}\Phi* \nabla^{c}\Phi + \sum \nabla^{m} T * \nabla^a T * \nabla^{b}R * \nabla^{c}\Phi .
\end{align*}
Third term: separating the $i=0$ from the others, we get
\begin{align*}
   | \sum \nabla^{m} T *\nabla^{m-i} T * \nabla^i R | \leq C K^2 |\nabla^{m} T|^2 + 
   C K^3 t^{-\tfrac m2}|\nabla^{m} T| .
\end{align*}
By induction we show that $|\nabla^i \Phi |\leq C \sum_{j=1}^i K^j t^{-\tfrac{j-i}2}$.\\ 
Fifth term: Separating the cases where $a$ or $b$ equals $m$ and using $K^2 T\leq 1$ we have
\begin{align*}
    |\sum \nabla^{m} T *\nabla^a T*\nabla^b T* \nabla^c T* \nabla^d \Phi|\leq C K^2 |\nabla^{m} T|^2 + C K^3 |\nabla^{m} T|t^{-\tfrac m2} .
\end{align*}
Seventh term: Using $K^2 T\leq 1$ we have
\begin{align*}
    |\sum \nabla^{m} T *\nabla^{m+1-i} R *\nabla^i \Phi|\leq C K^3 |\nabla^{m} T|t^{-\tfrac m2} .
\end{align*}
Sixth term: Separate the case $i=m$ from the others and use $K^2 T\leq 1$ to obtain
\begin{align*}
    |\sum \nabla^{m} T *\nabla^i T* \nabla^{m-i} R|\leq C K^2 |\nabla^{m} T|^2 + C K^3 |\nabla^{m} T|t^{-\tfrac m2} .
\end{align*}
Eighth term: Separate the $b=m$ term from the others, use $\nabla^{i+1}\Phi = \nabla^i T + \text{lot}$ and $K^2 T\leq 1$ to obtain
\begin{align*}
    |\sum \nabla^{m} T *\nabla^a R * \nabla^{b+1}\Phi* \nabla^{c}\Phi|\leq C K^2 |\nabla^{m} T|^2 + C K^2 |\nabla^{m} T|t^{-\tfrac{m+1}2} .
\end{align*}
Ninth term: Separate the $a=m$ term from the others and use $K^2 T\leq 1$ to obtain
\begin{align*}
    |\sum \nabla^{m} T * \nabla^a T * \nabla^{b}R * \nabla^{c}\Phi|\leq C K^2 |\nabla^{m} T|^2 + C K^3 |\nabla^{m} T|t^{-\tfrac m2} .
\end{align*}
Fourth term: Using $\nabla(T * \Phi) = \nabla T * \Phi + T*T * \Phi$,
\begin{align*}
&|\nabla^{m} T*\nabla^m(\nabla T * T * \Phi)| \leq |\nabla^{m} T*\nabla^{m+1}T * T * \Phi)|
  +  |\nabla^{m} T*\nabla^{m}T * \nabla(T * \Phi)| \\
  &+ |\nabla^{m} T*\sum_{i=2}^{m-1}\nabla^{m+1-i}T * \nabla^i (T * \Phi)| +
  |\nabla^{m} T*\nabla T * \nabla^m(T * \Phi)| \\
&  \leq CK |\nabla^{m} T||\nabla^{m+1} T| + C |\nabla^{m} T|^2(K t^{-\tfrac 12}+ K^2) 
  + C |\nabla^{m} T| \sum_{i=2}^{m-1}Kt^{-\tfrac {m+1-i}2}\sum_{j=0}^i K t^{-\tfrac {i-j}2}\sum_{k=1}^{j} K^k t^{\tfrac {k-j}2}\\
&  + C |\nabla^{m} T| K t^{-\tfrac 12} (|\nabla^{m} T| + \sum_{j=1}^m K t^{-\tfrac {m-i}2}\sum_{k=1}^{i} K^k t^{\tfrac {k-i}2}\\
 & \leq CK|\nabla^{m} T||\nabla^{m+1} T| + C|\nabla^{m} T|^2 (K t^{-\tfrac 12}+ K^2) + 
  C K^2 |\nabla^{m} T| t^{-\tfrac{m+1}2} .
\end{align*}
In conclusion
\begin{align*}
    &\frac{\partial}{\partial t} |\nabla^m T|^2 \leq \Delta |\nabla^m T|^2 - 2|\nabla^{m+1} T|^2 
    + C K^2 |\nabla^{m} T|^2 + C K^3 t^{-\tfrac{m}2}|\nabla^{m} T| + C K^2 t^{-\tfrac{m+1}2}|\nabla^{m} T| \\
    &+ CK|\nabla^{m} T||\nabla^{m+1} T|+  C K t^{-\tfrac{1}2}|\nabla^{m} T|^2 ,
\end{align*}
which is exactly Equation (3.32) in \cite{DGK}.

As the rest of the proof completely relies on this equation, it can be read in \cite{DGK}.
\end{proof}

\section{Analysis of the flow II}

Let $(M, g)$ be a complete Riemannian manifold. For $x_0\in M$, let $u$ be the fundamental solution of the backward heat equation, starting with the delta function at $x_0$ \cite{Hamilton1993}:
\begin{align*}
    & \Big( \frac{\partial}{\partial t} + \Delta \Big) u = 0, \quad \lim_{t \to t_0} u = \delta_{x_0}
\end{align*}
and set $u= \frac{e^{-f}}{\big( 4\pi (t_0 -t) \big)^4}.$

For a solution $\{\Phi(t)\}_{t\in [0, t_0)}$ of the harmonic $\mathrm{Spin}(7)$-flow on $(M,g)$, we define the function
\begin{align}\label{thetadefn}
  \Theta_{(x_0, t_0)} (\Phi(t)) = (t_0 -t) \int_M |T_{\Phi(t)}|^2 u \, \vol.  
\end{align}

We start off with a derivation of the function $\Theta\circ\Phi$ and give a little bit more details of the proof than in \cite[Lemma 5.1]{DLSE}.
\begin{lemma} \label{eqlemma4.4}
\begin{align*}
\frac{d}{d t} \Theta &= - 2(t_0 -t) \int_M  |\Div T - f \lrcorner T|^2 u   \\
 & - 2(t_0 -t) \int_M  \big( \nabla_m \nabla_l u - \frac{\nabla_m u \nabla_l u}{u}+ \frac{ug_{ml}}{2(t_0 -t)} \big) T_{m;is}T_{l;is}\\
 & - (t_0 -t) \int_M u R_{mlis} (2T_{l;ir}T_{m;rs} - 2 T_{m;ir}T_{l;rs} + \tfrac{1}{4} R_{mlis} - \tfrac18 R_{mlab} \Phi_{abis} ) \\
 & - 2(t_0 -t) \int_M T_{m;is}  u \nabla_l R_{mlis} .
 \end{align*}
\end{lemma}

\begin{proof}
By direct computation, we have
\begin{align*}
   & \frac{d}{d t} \Theta = \int_M (t_0 -t)  u \frac{\partial}{\partial t} |T|^2    - |T|^2 u    +    (t_0 -t) |T|^2 \frac{\partial}{\partial t} u \,  v_g\\
    &= \int_M (t_0 -t)  u \frac{\partial}{\partial t} |T|^2    - |T|^2 u - (t_0 -t) |T|^2 \Delta u  \, v_g\\
     &= \int_M 2(t_0 -t)  u T_{m;is}\frac{\partial}{\partial t} T_{m;is}    - |T|^2 u - (t_0 -t) |T|^2 \Delta u  \, v_g\\
     &= \int_M 2(t_0 -t)  u T_{m;is}\Big(\nabla_r T_{r;ip} T_{m;ps} - \nabla_r T_{r;sp} T_{m;pi} + \pi_7 (\nabla_m (\nabla_r T_{r;is})) )\Big)    - |T|^2 u \\
     &- (t_0 -t) |T|^2 \Delta u  \, v_g ,
\end{align*}
but
\begin{enumerate}
    \item $ T_{m;is} \nabla_r T_{r;ip} T_{m;ps} =0$
    because $T_{m;is} T_{m;ps}$ is symmetric in $(i,p)$ and $\nabla_r T_{r;ip}$ is skew-symmetric in $(i,p)$;
    \item $ T_{m;is} \nabla_r T_{r;sp} T_{m;pi} =0$
    because $T_{m;is} T_{m;pi}$ is symmetric in $(s,p)$ and $\nabla_r T_{r;sp}$ is skew-symmetric in $(s,p)$;
    \item $T_{m;is} \in \Lambda^{2}_{7}$,
\end{enumerate}
therefore
$$
\frac{d}{d t} \Theta = \int_M 2(t_0 -t)  u T_{m;is} (\nabla_m (\nabla_r T_{r;is})) - |T|^2 u - (t_0 -t) |T|^2 \Delta u  \, v_g .
$$
Integrating by parts and using the Bianchi identity, we have
\begin{align*}
   & \frac{d}{d t} \Theta =  \int_M - 2(t_0 -t)  u \nabla_m T_{m;is}  \nabla_r T_{r;is} - 2(t_0 -t) \nabla_m u  T_{m;is}  \nabla_r T_{r;is} - |T|^2 u \\
   &+ 2 (t_0 -t) T_{m;is} \nabla_l T_{m;is} \nabla_l u  \, v_g \\
   & = \int_M - 2(t_0 -t) \Big( |\Div T|^2 u + \nabla_m u  T_{m;is}  \nabla_r T_{r;is}\Big) - |T|^2 u + 2 (t_0 -t) T_{m;is} \nabla_l T_{m;is} \nabla_l u  \, v_g \\
    & = \int_M - 2(t_0 -t) \Big( |\Div T|^2 u + \nabla_m u  T_{m;is}  \nabla_r T_{r;is}\Big) - |T|^2 u  \\
    &+ 2 (t_0 -t) T_{m;is} \nabla_l u \Big( \nabla_m T_{l;is}
     + 2 T_{l;ir}T_{m;rs} - 2 T_{m;ir}T_{l;rs} + \tfrac{1}{4} R_{mlis} - \tfrac18 R_{mlab} \Phi_{abis}
    \Big) \, v_g \\
& = \int_M - 2(t_0 -t) \Big( |\Div T|^2 u + \nabla_m u  T_{m;is}  \nabla_r T_{r;is}\Big) - |T|^2 u  + 2 (t_0 -t) \Big(  T_{m;is} \nabla_l u \nabla_m T_{l;is} \\
& \qquad + 2 T_{m;is} \nabla_l u T_{l;ir}T_{m;rs} - 2 T_{m;is} \nabla_l u T_{m;ir}T_{l;rs} + T_{m;is} \nabla_l u\tfrac{1}{4} R_{mlis} - \tfrac18 T_{m;is} \nabla_l u R_{mlab} \Phi_{abis} \Big) \, v_g ,
\end{align*}
but, as previously,
\begin{enumerate}
    \item $T_{m;is} T_{l;ir}T_{m;rs} =0 $;
    \item $T_{m;is} T_{m;ir}T_{l;rs} =0$;
    \item $T_{m;is} (\tfrac{1}{4} R_{mlis} - \tfrac18 R_{mlab} \Phi_{abis}) = T_{m;is} R_{mlis}$,
\end{enumerate}
so
\begin{align*}
   & \frac{d}{d t} \Theta = 
 \int_M - 2(t_0 -t) \Big( |\Div T|^2 u + \nabla_m u  T_{m;is}  \nabla_r T_{r;is}\Big) - |T|^2 u  \\ 
 & \qquad + 2 (t_0 -t) \Big(  T_{m;is} \nabla_l u \nabla_m T_{l;is}
 + T_{m;is} \nabla_l u R_{mlis} \Big) \, v_g  \\
 & = \int_M - 2(t_0 -t) \Big( |\Div T|^2 u + \nabla_m u  T_{m;is}  \nabla_r T_{r;is}\Big) - |T|^2 u  \\
 & - 2(t_0 -t) \int_M \nabla_m T_{m;is} \nabla_l u T_{l;is} + T_{m;is} \nabla_m \nabla_l u T_{l;is}\\
 & - 2(t_0 -t) \int_M \nabla_l T_{m;is}  u R_{mlis} + T_{m;is}  u \nabla_l R_{mlis} \\
 &= \int_M - 2(t_0 -t) \Big( |\Div T|^2 u + 2 \nabla_m u  T_{m;is}  \nabla_r T_{r;is}\Big) - |T|^2 u  \\
 & - 2(t_0 -t) \int_M T_{m;is} \nabla_m \nabla_l u T_{l;is} \\
&  - (t_0 -t) \int_M u R_{mlis} (\nabla_l T_{m;is} - \nabla_m T_{l;is}) \\
 & - 2(t_0 -t) \int_M T_{m;is}  u \nabla_l R_{mlis} 
 \end{align*}
 \begin{align*}
 &= \int_M - 2(t_0 -t) \Big( |\Div T|^2 u + 2 \nabla_m u  T_{m;is}  \nabla_r T_{r;is}\Big) - |T|^2 u  \\
 & - 2(t_0 -t) \int_M T_{m;is} \nabla_m \nabla_l u T_{l;is}\\
 & - (t_0 -t) \int_M u R_{mlis} (2T_{l;ir}T_{m;rs} - 2 T_{m;ir}T_{l;rs} + \tfrac{1}{4} R_{mlis} - \tfrac18 R_{mlab} \Phi_{abis} ) \\
 & - 2(t_0 -t) \int_M T_{m;is}  u \nabla_l R_{mlis} \\
 &= \int_M - 2(t_0 -t) \Big( |\Div T|^2 u - 2 \langle \Div T , \nabla f \lrcorner T \rangle  u \Big)  \\
 & - 2(t_0 -t) \int_M  \big( \nabla_m \nabla_l u + \frac{ug_{ml}}{2(t_0 -t)} \big) T_{m;is}T_{l;is}\\
 & - (t_0 -t) \int_M u R_{mlis} (2T_{l;ir}T_{m;rs} - 2 T_{m;ir}T_{l;rs} + \tfrac{1}{4} R_{mlis} - \tfrac18 R_{mlab} \Phi_{abis} ) \\
 & - 2(t_0 -t) \int_M T_{m;is}  u \nabla_l R_{mlis} \\
 &= \int_M - 2(t_0 -t) \Big( |\Div T|^2 u - 2 \langle \Div T , \nabla f \lrcorner T \rangle  u + |f \lrcorner T|^2 u \Big)  \\
 & - 2(t_0 -t) \int_M  \big( \nabla_m \nabla_l u - \frac{\nabla_m u \nabla_l u}{u}+ \frac{ug_{ml}}{2(t_0 -t)} \big) T_{m;is}T_{l;is}\\
 & - (t_0 -t) \int_M u R_{mlis} (2T_{l;ir}T_{m;rs} - 2 T_{m;ir}T_{l;rs} + \tfrac{1}{4} R_{mlis} - \tfrac18 R_{mlab} \Phi_{abis} ) \\
 & - 2(t_0 -t) \int_M T_{m;is}  u \nabla_l R_{mlis} ,
\end{align*}
and we obtain the formula we wish for.
\end{proof}

\begin{theorem}[almost monotonicity formula] \cite[Theorem 5.2]{DLSE}
\label{thm:almostmon} \\
Let $\{\Phi(t)\}$ be a solution of the harmonic $\mathrm{Spin}(7)$-flow \eqref{eq: Har Spin(7) Flow} on $(M^8,g)$.
\begin{enumerate}
    \item If $M$ is compact, then, for any $0<\tau_1 < \tau_2 < t_0$, there exist $K_1$, $K_2>0$ depending only on the geometry of $(M,g)$ such that
$$
\Theta(\Phi(\tau_2)) \leq K_1 \Theta(\Phi(\tau_1)) + K_2 (\tau_1 -\tau_2) (E(0) +1) .
$$
    \item When $(M,g) = (\bR^8, g_{\mathrm{Eucl}})$, then, for any $x_0 \in \bR^8$ and $0\leq \tau_1 < \tau_2$ we have
$$
\Theta(\Phi(\tau_2)) \leq \Theta(\Phi(\tau_1)) .
$$
\end{enumerate}
\end{theorem}

\begin{proof}
We sketch the proof following \cite[Theorem 5.3]{DGK}.
\begin{enumerate}
\item The following equation is a direct adaptation of \cite[Lemma 5.2]{DGK}, using the $\mathrm{Spin}(7)$-Bianchi identity \eqref{spin7bianchi}:
\begin{align} \label{lem: d/dt of Theta}
    \frac{d}{d t} \Theta_{(x_0, t_0)}(\Phi(t)) 
    =& - 2(t_0 -t) \int_M  |\Div T - \del f \lrcorner T|^2 u \vol \notag\\
    & - 2(t_0 -t) \int_M  \big( \nabla_m \nabla_l u - \frac{\nabla_m u \nabla_l u}{u}+ \frac{ug_{ml}}{2(t_0 -t)} \big) T_{m;is}T_{l;is} \vol  \notag\\
    & - (t_0 -t) \int_M u R_{mlis} (2T_{l;ir}T_{m;rs} - 2 T_{m;ir}T_{l;rs} + \tfrac{1}{4} R_{mlis} - \tfrac18 R_{mlab} \Phi_{abis} ) \vol \notag\\
    & - 2(t_0 -t) \int_M T_{m;is}  u \nabla_l R_{mlis} \vol. 
\end{align}
    \item  The third and fourth terms of  Lemma~\eqref{eqlemma4.4} are bounded by
$$
C( 1 + \Theta(\Phi(t))),
$$
due to the bounded geometry of $(M,g)$, Young's inequality and $\int_M  u \, \vol =1$.

For the second term of Lemma~\eqref{eqlemma4.4}, use \cite{hamilton-compactness} and the decreasing of 
$E(\Phi(t))$ along the harmonic $\mathrm{Spin}(7)$-flow to bound it by
$$
C \big(E(\Phi(0)) + \log \frac{B}{(t_0 -t)^4} \Theta(\Phi(t))\big),
$$
so that
\begin{align*}
\frac{d}{dt} \Theta(\Phi(t))  \leq& - 2(t_0 -t) \int_M  |\Div T - \del f \lrcorner T|^2 u \\
& +C_1 \Big(1 + \log \big( \frac{B}{(t_0 -t)^4}\big) \Big)\Theta(\Phi(t)) 
+ C_2 ( 1 + E(\Phi(0))) \vol.
\end{align*}
To control the logarithmic term, let $\xi(t)$ be any function satisfying
$$
\xi'(t) = 1 + \log \frac{B}{(t_0 - t)^4}.
$$
The claim is then obtained by integration over  $[t_0 -1 , t_0[\,$ of
$$
\frac{d}{dt} \Big[ e^{-C_1 \xi(t)} \Theta(\Phi(t))\Big] \leq K \big(E(\Phi(0)) + 1 \big).
$$
\item On $(M^8,g) = (\bR^8, g_{\mathrm{Eucl}})$, the backward heat kernel is
$$
u(x,t) = \frac{1}{(4\pi(t_0-t))^4}
    \exp\left\{-\frac{|x-x_0|^2}{4(t_0-t)}\right\}
$$
so indeed
$\frac{d}{dt} \Theta(\Phi(t))  \leq 0.$
\end{enumerate}
\end{proof}

There exists a more direct and cost-effective to obtain a decreasing quantity from $|T|^2$, though its importance remains uncertain.
\begin{lemma}[a simpler monotonicity formula]
Put $\epsilon(t) = |T|^2$ and consider the function 
$$
Z(t) = (\rT-t) \int_M \epsilon k,
\quad
0\leq t< \rT,
$$
where $k$ is any (positive) solution of the backward heat equation $\partial_t k = - \Delta k$ on $M_{\rT}$.
Then 
$$ Z(t) \leq Z(0) e^{Ct}$$
for $0\leq t \leq \delta$ (from DTE).
\end{lemma}
\begin{proof}
Since 
$$
\partial_t Z 
= - \int_M \epsilon k 
  + (\rT-t) \int_M k \partial_t\epsilon +  \epsilon \partial_t k.
$$
By self-adjointness of the Laplacian and the `reaction-diffusion' Bochner formula of Equation~\eqref{bochner}, the second integral satisfies the following upper bound:
\begin{align*}
    \int_M k \partial_t\epsilon +  \epsilon \partial_t k 
    &= \int_M k \partial_t\epsilon 
      - \epsilon \Delta k 
    = \int_M k \partial_t\epsilon 
       - k \Delta \epsilon  \\
    &= \int_M k (\partial_t\epsilon - \Delta \epsilon)  \\
    &\leq \int_M k (C_1 \epsilon + C_2 \epsilon^2 ), 
\end{align*}
and therefore
$$
\partial_t Z \leq  C_1 Z(t) + (\rT-t) \int_M k\epsilon (C_2 \epsilon),
\quad
0\leq t < \rT.
$$
Then by DTE, we have 
$$
 C_2 \epsilon(x,t) \leq C_2 \cT(t) \leq 2C_2 \cT(0) = C_0
$$
so
\begin{align*}
\partial_t Z &\leq  C_1 Z(t) + (\rT-t) \int_M k\epsilon (  C_2 \epsilon  )\\
&\leq  C_1 Z(t) + C_0 (\rT-t) \int_M k\epsilon \\
&\leq C Z(t)
\end{align*}
so 
$$ Z(t) \leq Z(0) e^{Ct}.$$
\end{proof}

\begin{definition}
Let $(M^8, \Phi, g)$ be a compact manifold with a $\mathrm{Spin}(7)$-structure. Let $u_{(x,t)}(y,s)=u^g_{(x,t)}(y,s)$ be the backwards heat kernel, starting from $\delta{(x,t)}$ as $s\rightarrow t$. For $\sigma>0$ we define
\begin{align}\label{eq:entropyeqn}
\lambda(\Phi, \sigma) = \underset{(x,t)\in M\times (0, \sigma]}{\textup{max}} \left\{ t\int_M |T_{\Phi}|^2(y) u_{(x,t)}(y,0) \vol\right \}.     
\end{align}
\end{definition}
One should think of $\sigma$ as the ``scale'' at which we are analyzing the flow. Since $M$ is compact, the maximum in \eqref{eq:entropyeqn} is achieved.

We can now state the $\varepsilon$-regularity theorem for the harmonic $\mathrm{Spin}(7)$-flow.

\begin{theorem}[$\varepsilon$-regularity]\label{thm:epsilonreg} \cite[Theorem 5.5]{DLSE}
    Let $(M^8,g)$ be compact and $E_0>0$. There exist $\varepsilon,\ \bar{\rho}>0$ such that, for every $\rho \in (0, \bar{\rho}]$, there exist $r\in (0, \rho)$ and $C<\infty$ such that the following holds:
    
    Suppose $\{\Phi(t)\}_{t\in [0, t_0)}$ is a solution of the harmonic $\mathrm{Spin}(7)$-flow \eqref{eq: Har Spin(7) Flow}, with induced metric $g$, satisfying $E(\Phi(0))\leq E_0$. Whenever 
\begin{align*}
    \Theta_{(x_0, t_0)}(\Phi(t_0-\rho^2)) < \varepsilon,    
    \quad \text{for some }  x_0\in M,
\end{align*}
then, setting $\Lambda_r(x,t) = \textup{min}\ \left(1-r^{-1}d_g(x_0, x), \sqrt{1-r^{-2}(t_0-t)}    \right)$, we have
\begin{align*}
    \Lambda_r(x, t)|T_{\Phi}(x,t)|\leq \frac{C}{r},    
    \quad\forall \ 
    (x,t)\in B(x_0, r)\times [t_0-r^2, t_0].
\end{align*}
\end{theorem}

An immediate corollary of the $\varepsilon$-regularity theorem is the following result, which states that if the entropy of the initial $\mathrm{Spin}(7)$-structure is small then the torsion is controlled at all times. Again, the proof is similar to \cite[Cor. 5.8]{DGK}.

\begin{corollary}[small initial entropy controls torsion]\cite[Corollary 5.6]{DLSE}
\label{cor:enttor}
    Let $\{\Phi(t)\}$ be a solution of the harmonic $\mathrm{Spin}(7)$-flow \eqref{eq: Har Spin(7) Flow} on compact $(M, g)$,  starting at $\Phi_0$. For every $\sigma>0$, there exist $\varepsilon, t_0>0$ and $C< \infty$ such that, if $\Phi_0$ induces $g$ and its entropy \eqref{eq:entropyeqn} satisfies
\begin{align*}
    \lambda(\Phi_0, \sigma) < \varepsilon,  
\end{align*}
then
\begin{align*}
    \underset{M}{\textup{max}}\ |T_{\Phi(t)}|\leq \frac{C}{\sqrt{t}}.    
\end{align*}

\end{corollary}

\begin{theorem}[small initial torsion gives long-time existence]\cite[Theorem 5.9]{DLSE}
\label{thm: smalltorconv}
    Let $(M,\Phi_0,g)$ be a compact  $\mathrm{Spin}(7)$-structure manifold. For every $\delta >0$, there exists $\varepsilon (\delta, g) >0$ such that, if $|T_{\Phi_0}| < \varepsilon$, then a harmonic $\mathrm{Spin}(7)$-flow \eqref{eq: Har Spin(7) Flow} starting at $\Phi_0$ exists for all time and converges subsequentially smoothly to a $\mathrm{Spin}(7)$-structure $\Phi_{\infty}$ such that 
$$
\Div T_{\Phi_{\infty}} =0, \quad
|T_{\Phi_{\infty}}| < \delta.
$$
\end{theorem}

\begin{proof}[Sketch of proof]
1) If if $|T_{\Phi_0}| < \varepsilon_0$ then by the (DTE) there exists $\delta >0$ such that 
\begin{align}
    t_* := \textup{max}\{ t\geq 0\ :\  |T_{\Phi(t)}| \leq 2\varepsilon_0\} 
    > \delta.   
\end{align}
2) If $t_* < \infty$ then the Shi-type estimates on $]t_*-\delta , t_*[$ would imply 
\begin{align}
    |\del T_{\Phi(t_*)}|< c_0.  
\end{align}
3) But our flow is the negative gradient so $E(\Phi(t_*))\leq E(\Phi_0)$ so we can invoke the interpolation lemma (which is a static result):
 If $|\del T|\leq C$ and no collapsing, i.e.
\begin{align*}
    \vol(B(x, r))\geq v_0r^8, \text{for } 0<r\leq 1,
\end{align*}
    for some constant $v_0(M,g)>0$, then, for every $\varepsilon>0$, there exists $\delta(\varepsilon, C, v_0)\geq 0$ such that, if $E(\Phi) < \delta$
    then $|T|<\varepsilon$.\\
4) Conclude taking $\varepsilon < \min(\varepsilon_0, \gamma_{2\varepsilon_0} $ so that $|\del T_{\Phi(t_*)}|< \varepsilon_0 $ implies $|T(t_*)| < 2\varepsilon_0$ which contradicts the maximality of $t_*$ and forces $t_* = +\infty$.\\
5) If $\Lambda$ is the first eigenvalue of the Laplacian on 2-forms we can easily show that (cf. \cite[Lemma 5.7]{DLSE})
\begin{align*}
    \frac{d^2}{dt^2}E(\Phi (t)) &\geq \int_M (\Lambda -3 |T|^2) |\Div T|^2 \vol, 
\end{align*}
so if $|T|^2\leq \frac{\Lambda}{6}$
$$
\frac{d}{dt} \int_M |\Div T_{\Phi(t)}|^2 \vol = - \frac{d^2}{dt^2} E(\Phi(t))\leq -\frac{\Lambda}{2}\int_M |\Div T_{\Phi(t)}|^2 \vol .
$$
If we take $\varepsilon < \min(\varepsilon_0, \gamma_{2\varepsilon_0}, \gamma_{\sqrt{\frac{\Lambda}{6}}})$ then we obtain the decay estimate
\begin{align}
    \int_M |\Div T_{\Phi(t)}|^2 \vol \leq e^{-\frac{\Lambda t}{2}}\int_M |\Div T_{\Phi(0)}|^2 \vol, 
    \quad\forall \ t\geq 0.
\end{align}
6) Take $s_1<s_2$ and integrate to obtain
\begin{equation} \label{eq:ltetorpf5}
\begin{aligned}
    \int_M |\Phi(s_2)-\Phi(s_1)| \vol 
    &\leq  \int_M \int_{s_1}^{s_2} \left| \partial_t \Phi(s)  \right| ds \vol
    = \int_{s_1}^{s_2} \int_M |\Div T_{\Phi(s)}| \vol ds\\
    &\leq c\int_{s_1}^{s_2} \left(\int_M |\Div T_{\Phi(s)}|^2 \vol \right)^{\frac{1}{2}} ds\\
    &\leq c \int_{s_1}^{s_2}  e^{-\frac{\Lambda s}{4}} ds. 
\end{aligned}
\end{equation}
7) $\Phi(t)$ converges in $L^1$ to $\Phi_{\infty}$. \\
8) The uniform bound on $T$ combined with Shi-type estimates gives estimates on all $|\nabla^m T|$ and smooth convergence to $\Phi_\infty$. \\
9) The exponential decay of the integrals implies that $\Div T_{\Phi_\infty}= 0$, and by the interpolation lemma, we also achieve that $|T_{\Phi_\infty} |<\delta$. 
\qed
\end{proof}

\begin{theorem}[small entropy gives long-time existence]\cite[Theorem 5.10]{DLSE}
\label{thm:smallent}
    On a compact  $\mathrm{Spin}(7)$-structure manifold $(M,\Phi_0,g)$, there exist constants $C_k(M,g) < +\infty$, such that the following holds. For each $\varepsilon>0$ and $\sigma >0$, there exists $\lambda_{\varepsilon} (g, \sigma) >0$ such that, if the entropy \eqref{eq:entropyeqn} satisfies 
\begin{align}
\label{smallenteq}
    \lambda(\Phi_0, \sigma) < \lambda_{\varepsilon},
\end{align}
    then the torsion becomes eventually pointwise small along the harmonic $\mathrm{Spin}(7)$-flow \eqref{eq: Har Spin(7) Flow} starting at $\Phi_0$. Therefore the flow exists for all time and subsequentially converges to a $\mathrm{Spin}(7)$-structure $\Phi_{\infty}$ such that 
    $$
    \Div T_{\Phi_{\infty}} =0, \quad
    |T_{\Phi_{\infty}}| < \varepsilon \qandq
    |\nabla^k T_{\Phi_{\infty}}| < C_k,
    \,\forall\  k\geq 1.
    $$
\end{theorem}
\begin{proof}[Sketch of proof]
1) $\lambda_{\varepsilon}$ small enough implies $ |T| \leq \frac{C}{\sqrt{t}}$ for all $t\leq \tau$.\\
2) Shi-type estimates imply $|\nabla T(\tau)| < C'$.\\
3) Interpolation lemma: $\forall \epsilon >0$, for small enough $\lambda_{\varepsilon}$, $|T(\tau)| < \epsilon$.\\
4) Small $|T(\tau)|$  implies long-time existence.\\
5) We conclude as with the previous result.
\qed
\end{proof}

Let $\varepsilon$ and $\bar{\rho}$ be the quantities from the $\varepsilon$-regularity Theorem~\ref{thm:epsilonreg}. We define the \emph{singular set} of the flow by 
\begin{align}
\label{eq:singsetdefn}
    S = \{ x\in M \ :\ \Theta_{(x, \tau)}(\Phi(\tau-\rho^2)) \geq \varepsilon, \ \textup{for\ all}\ \rho\in (0, \bar{\rho}]\}.    
\end{align}
The following lemma explains why $S$ is called the singular set of the flow.

\begin{lemma}
\label{lem:sing}
    The harmonic $\mathrm{Spin}(7)$-flow $\{\Phi(t)\}_{t\in [0, \tau)}$ restricted to $M\setminus S$ converges as $t\rightarrow \tau$, smoothly and uniformly away from $S$, to a smooth harmonic $\mathrm{Spin}(7)$-structure $\Phi(\tau)$ on $M\setminus S$. Moreover, for every $x\in S$, there is a sequence $(x_i, t_i)\to(x,\tau)$ such that 
\begin{align*}
    \lim_{i} |T_{\Phi}(x_i, t_i)|= \infty.    
\end{align*}
Thus, $S$ is indeed the singular set of the flow.
\end{lemma}
\begin{theorem}[Hausdorff measure of the singularity set] \cite[Theorem D]{DLSE}
\label{thm: singsize}
\begin{align}
\label{eq:singsize}
    E(\Phi_0)= \frac 12\int_M |T_{\Phi_0}|^2 \vol \leq E_0.  
\end{align}
    Suppose that the maximal smooth harmonic $\mathrm{Spin}(7)$-flow $\{\Phi(t)\}_{t\in [0,\tau)}$ starting at $\Phi_0$ blows up at time $\tau< +\infty$. Then, as $t\rightarrow \tau$, \eqref{eq: Har Spin(7) Flow} 
    converges smoothly to a $\mathrm{Spin}(7)$-structure $\mathrm{Spin}(7)_{\tau}$ away from a closed set $S$, with finite $6$-dimensional Hausdorff measure satisfying 
\begin{align*}
    \mathcal{H}^6(S) \leq CE_0,  
\end{align*}
    for some constant $C<\infty$ depending on $g$. In particular, the Hausdorff dimension of $S$ is at most 6.
\end{theorem}
\begin{proof}[Sketch of proof]
    The proof relies on the following computation:
    \begin{align*}
 \varepsilon \mathcal{H}^6(S) = \int_{S}\varepsilon d\mathcal{H}^6(x) &\leq \int_{S} \Theta(\Phi(\tau-\rho^2)) d\mathcal{H}^6(x) \\ 
 &\leq \int_{S}\int_M \rho^2 |T|^2 u\vol d\mathcal{H}^6(x)\\
& \leq \int_M \rho^2 |T|^2 u \vol\\
&\leq C \int_M |T|^2 \vol \\
& \leq CE_0.
\end{align*}
\qed
\end{proof}

\end{document}